\documentclass[12pt,reqno,twoside]{amsart}

\usepackage{amssymb,amsmath,amscd,enumerate,verbatim}
\usepackage[colorlinks=true,linkcolor=blue,citecolor=blue]{hyperref}
\usepackage[latin1]{inputenc}
\usepackage{color}

\usepackage[all]{xy}
\usepackage{pstricks, pst-3d}
\newcommand{\mm}{\mathfrak m}

\newcommand{\Z}{\mathbb{Z}}

\newcommand{\Q}{\mathbb{Q}}

\DeclareMathOperator{\pnt}{\raise 0.5mm \hbox{\large\bf.}}

\DeclareMathOperator{\Tor}{Tor}

\DeclareMathOperator{\reg}{reg}

\def\+#1{\relax\ifmmode\if\noexpand #1\relax \mathop{\kern
    0pt^+{#1}}\nolimits\else \kern 0pt^+\!#1 \fi\else$^*$#1\fi}

\newcommand{\exte}[1]{\left<#1\right>}

\let\phi=\varphi
\let\:=\colon

\newtheorem{thm}{\bf Theorem}[section]
\newtheorem{lem}[thm]{\bf Lemma}

\newtheorem{prop}[thm]{\bf Proposition}

\theoremstyle{definition}

\theoremstyle{plain}
\newtheorem*{thm*}{Theorem}
\newtheorem*{lem*}{Lemma}
\newtheorem*{cor*}{Corollary}
\newtheorem*{claim*}{Claim}
\newtheorem*{defn*}{Definition}

\theoremstyle{remark}
\newtheorem{rem}[thm]{Remark}

\newtheorem{ex}[thm]{Example}

\numberwithin{equation}{section}
\title{Koszul hypersurfaces over the exterior algebras}

\author{Hop D. Nguyen}
\address{Ernst-Abbe-Platz 5, Appartment 605, 07743 Jena, Germany}
\email{ngdhop@gmail.com}
\thanks{This research was supported by the CARIGE foundation.}
\date{\today}

\begin{document}

\begin{abstract}
We prove that if $E$ is an exterior algebra over a field, $h$ is a quadratic form, then $E/(h)$ is Koszul if and only if $h$ is a product of two linear forms. 
\end{abstract}

\maketitle

\section{The main result}
Let $k$ be a field. In this note, we consider graded $k$-algebras which are quotients of the polynomial or the exterior algebras with the standard grading. Let $R$ be such a standard graded over $k$. We say that $R$ is a {\it Koszul algebra} if the residue field $k$ has a linear free resolution over $R$. For recent surveys on Koszul algebras, the reader may consult \cite{CDR}, \cite{Fr}.

While the Koszul property was studied intensively for commutative algebras, it is less well-known for quotients of exterior algebras. Any exterior algebra is Koszul because the Cartan complex is the linear free resolution for the residue field; see \cite[Section 2]{AHH}. The main result of this note says, quite unexpectedly, that exterior Koszul hypersurfaces are reducible. It answers in the positive an intriguing question of Phong Thieu in his thesis \cite[Question 5.2.8]{Th}.
\subsection*{Notation}
If $V$ is a $k$-vector space of dimension $n$ with a fixed $k$-basis $e_1,\ldots,e_n$, then the notation $k\exte{e_1,\ldots,e_n}$ denotes the exterior algebra $\bigwedge V$. It is a graded-commutative $k$-algebra with the grading induced by $\deg e_i=1$ for $1\le i\le n$, so that for all homogeneous elements $a,b\in E=k\exte{e_1,\ldots,e_n}$, we have
\[
a\wedge b=(-1)^{(\deg a)(\deg b)}b\wedge a
\]
and
\[
a\wedge a=0,
\]
if $\deg a$ is odd. For simplicity, we will denote $a\wedge b$ simply by $ab$. 

The main result of this note is
\begin{thm}
\label{main}
Let $E=k\exte{e_1,\ldots,e_n}$ be an exterior algebra (where $n\ge 1$) and $h$ a homogeneous form. Then $E/(h)$ is Koszul if and only if $h$ is a reducible quadratic form, namely $h$ is a product of two linear forms.
\end{thm}
This is in stark contrast with the fact that any commutative quadratic hypersurface is Koszul.

\section{Proof of the main result}
The main work in our result is done by the following 
\begin{thm}
\label{thm_rank2}
Let $h=e_1f_1+e_2f_2+\cdots+e_nf_n$ be a quadratic form in the exterior algebra in $2n$ variables $E=k\left<e_1,\ldots,e_n,f_1,\ldots,f_n\right>$, where $n\ge 2$. Then $E/(h)$ is not Koszul.
\end{thm}
Let $M$ be a finitely generated graded $R$-module. If $R$ is a quotient of an exterior algebra, this means that $M$ is finitely generated,  graded, and on $M$ there are left and right $R$-actions such that 
$$
am=(-1)^{(\deg a)(\deg m)}ma
$$
for every homogeneous elements $a\in R, m\in M$.

Denote $\beta^R_{i,j}(M)=\dim_k \Tor^R_i(M,k)_j$ the $(i,j)$-graded Betti number of $M$, where $i,j\in \Z$. The $i$th total Betti number of $M$ is $\beta^R_i(M)=\dim_k \Tor^R_i(M,k)$. We let $t^R_i(M)=\sup\{j:\beta^R_{i,j}(M)\neq  0\}$ for each $i\in \Z$. The last theorem will be deduced from the following analog of \cite[Main Theorem, (1)]{ACI1}, which is about syzygies of Koszul algebras in the commutative setting.
\begin{prop}
\label{Koszul_syzygy}
Let $Q\to R$ be a surjection of Koszul algebras. Let $K_{\pnt}$ be the minimal free resolution of $k$ over $Q$. Denote by $Z_i$ the $i$-th cycle of the complex $K_{\pnt}\otimes_Q R$. Then $\reg_R Z_i\le \reg_R Z_{i-1}+2$ for all $i\ge 0$. In particular, $t_i^Q(R)\le 2i$ for all $i\ge 0$.
\end{prop}
\begin{proof}
The proof is similar to that of \cite[Lemma 2.10]{Con}. Denote by $B_{\pnt},H_{\pnt}$ the boundary and homology of the complex $C_{\pnt}=K_{\pnt}\otimes_Q R$. For each $i>0$, we have short exact sequences
\begin{align*}
0\to Z_i \to C_i \to B_{i-1}\to 0,\\
0 \to B_{i-1} \to Z_{i-1} \to H_{i-1} \to 0.
\end{align*}
Therefore combining with the fact that $K$ is a linear resolution, we get
\begin{align*}
\reg_R Z_i &\le \max\{\reg_R C_i,\reg_R B_{i-1}+1\}\\
           &\le \max\{i,\reg_R Z_{i-1}+1,\reg_R H_{i-1}+2\}.
\end{align*}
Let $\mm$ be the graded maximal ideal of $R$. Since $H_i\cong \Tor^Q_i(k,R)$, we have $\mm H_i=0$ for all $i\ge 0$. Now $\reg_R k=0$ since $R$ is a Koszul algebra, hence $\reg_R H_{i-1}=t^R_0(H_{i-1})\le t^R_0(Z_{i-1})\le \reg_R Z_{i-1}$. Hence
\[
\reg_R Z_i \le \max\{i,\reg_R Z_{i-1}+2\}.
\]
As $Z_i\subseteq C_i$ and $t^R_0(C_i)=i$, we get $\reg_R Z_i\ge t^R_0(Z_i)\ge i$. Therefore $\reg_R Z_i\le \reg_R Z_{i-1}+2$. By induction, $\reg_R Z_i\le 2i$ for all $i\ge 0$. For the second statement, we have $t^Q_i(R)=t^Q_0(H_i)=t^R_0(H_i)\le t^R_0(Z_i) \le \reg_R Z_i\le 2i$.
\end{proof}

\begin{lem}
\label{lem_second_syzygy}
Let $R=E/(h)$ where $E=k\left<e_1,\ldots,e_n,f_1,\ldots,f_n\right>$ and $h=e_1f_1+e_2f_2+\cdots+e_nf_n$. Then $\beta_{2,n+2}^E(R)\ne 0$.
\end{lem}
\begin{proof}
This follows since the identity $e_1\cdots e_nh=0$ gives a minimal second syzygy of degree $n+2$ of $R$ as an $E$-module; see the proof of \cite[Theorem 1.2]{Mc}.
\end{proof}
\begin{proof}[Proof of Theorem \ref{thm_rank2}]
For $n\ge 3$, from Lemma \ref{lem_second_syzygy} and Proposition \ref{Koszul_syzygy}, we get the result. For $n=2$ then $E=k\left<e_1,e_2,f_1,f_2\right>$, $h=e_1f_1+e_2f_2$. Denote $A=E/(h)$, and 
\[
P^A_k(t)=\sum_{i=0}^{\infty}\beta^R_i(k)t^i\in \Z[[t]]
\]
the Poincar\'e series of $k$. If $A$ were Koszul, we obtain an equality
\[
P^A_k(t)H_A(-t)=1,
\]
where $H_A(t)$ is the Hilbert series of $A$. Note that $H_A(t)=1+4t+5t^2$, hence
\[
1/H_A(-t)=1+4t+11t^2+24t^3+41t^4+44t^5-29t^6-\cdots
\]
which cannot be the non-negative series $P^A_k(t)$. Hence $A$ is not Koszul.
\end{proof}

\begin{ex}
Consider $E=\Q[e_1,e_2,e_3,f_1,f_2,f_3]$ and $h=e_1f_1+e_2f_2+e_3f_3$. The Betti table of the minimal free resolution of $\Q$ over $E/(h)$ is given by Macaulay2 \cite{GS} as follow
\begin{verbatim}
            0 1  2  3   4    5    6     7
     total: 1 6 22 76 302 1272 5189 20614
         0: 1 6 22 62 148  314  610  1106
         1: . .  .  .   .    .    .     .
         2: . .  . 14 154  958 4383 16372
         3: . .  .  .   .    .    .     .
         4: . .  .  .   .    .  196  3136
\end{verbatim}
\end{ex}

\begin{proof}[Proof of Theorem \ref{main}]
We can assume that $h$ is a quadratic form.

The ``if" part is easy and was pointed out in \cite[Theorem 5.1.5]{Th}; we include an argument here. There is nothing to do if $n=1$ or $h=0$, so we can assume that $n\ge 2$ and $h\neq 0$. If $h$ is reducible, by a suitable change of coordinates, we can assume that $h=e_1e_2$. Then $E/(h)$ has a Koszul filtration in the sense of \cite{CTV} hence it is Koszul. 

The ``only if" part: if $h$ is not reducible, by a suitable change of coordinates, we can assume that $h=e_1e_2+e_3e_4+\cdots+e_{2i-1}e_{2i}$ for some $2\le i\le n/2$. Then $k\exte{e_1,\ldots,e_{2i}}/(h)$ is not Koszul by Theorem \ref{thm_rank2}, and it is an algebra retract of $E/(h)$. Therefore $E/(h)$ is also not Koszul.
\end{proof}

\section{Final remarks}
\begin{rem}
The theory of Gr\"obner basis was extended to the exterior algebras by Aramova, Herzog and Hibi in \cite{AHH}. It is also true as in the commutative case that, if $J$ is a homogeneous ideal in the exterior algebra $E$ and $J$ has quadratic Gr\"obner basis with respect to some term order on $E$, then $E/J$ is Koszul (see \cite[Theorem 5.1.5]{Th}).

Note however that the ideal $I=(e_1f_1+e_2f_2+\cdots+e_nf_n)$ of $E=k\left<e_1,\ldots,e_n,f_1,\ldots,f_n\right>$ (where $n\ge 2$) does not have a quadratic Gr\"obner basis in the natural coordinates $e_1,\ldots,e_n,f_1,\ldots,f_n$. Indeed, we can assume that $e_1f_1$ is the initial form of $e_1f_1+e_2f_2+\cdots+e_nf_n$. Now clearly $e_1(e_2f_2+\cdots+e_nf_n)\in I$, and its initial form is not divisible by $e_1f_1$, so any Gr\"obner basis of $I$ must contain a cubic form.
\end{rem}

\begin{rem}
(i) Note that the ideal $J=(e_1e_2-f_1f_2,e_1f_1-e_2f_2)$ of $E=k\left<e_1,e_2,f_1,f_2\right>$ is generated by two irreducible quadratic forms, but $E/J$ is Koszul. In fact, it has a Koszul filtration; see \cite[Example 5.2.2(ii)]{Th}. On the other hand, it is shown in {\it loc.~ cit.} that $J$ has no quadratic Gr\"obner basis in the natural coordinates.

(ii) Note that if $E=k\exte{e_1,\ldots,e_n}$ be an exterior algebra and $I$ be an ideal generated by reducible quadratic forms then $E/I$ is not Koszul in general. For example, $E=\Q\exte{e_1,\ldots,e_5}$ and $I=(e_1e_2,e_3e_4,(e_1+e_3)e_5)$ then $E/I$ is not Koszul, since the Betti table of $\Q$ over $E/I$ is
\begin{verbatim}
             0 1  2  3   4
      total: 1 5 18 57 171
          0: 1 5 18 56 160
          1: . .  .  1  11
\end{verbatim}
(iii) Many results on the exterior algebras have analogues over commutative complete intersections; see, e.g., \cite{AAH}, \cite{AHH}. However, if $S=k[x_1,\ldots,x_n]/(x_1^2,\ldots,x_n^2)$ and $h$ a reducible quadratic form, $S/(h)$ is not Koszul in general. For example, take $S=\Q[x,y,z,t]/(x^2,y^2,z^2,t^2)$ and $h=x(y+z+t)$ then $S/(h)$ is not Koszul since the resolution of $\Q$ over $S/(h)$ is
\begin{verbatim}
             0 1  2  3  4
      total: 1 4 11 27 66
          0: 1 4 11 25 51
          1: . .  .  2 15
\end{verbatim}
\end{rem}

\end{document}